\documentclass[reqno,12pt]{amsart} %
\usepackage{amsmath,amstext, amsthm, amssymb,amsfonts}
\usepackage{fancybox}
\usepackage{mathtools}
 \usepackage[top=1 in,bottom=1in, left=1 in, right= 1 in]{geometry}

\numberwithin{equation}{section}

\newtheorem{theorem}{Theorem}[section]
\newtheorem{proposition}[theorem]{Proposition}

\newtheorem{corollary}[theorem]{Corollary}
\theoremstyle{definition}

\newtheorem{remark}[theorem]{Remark}

\newcount\mins \newcount\hours \hours=\time \mins=\time
\divide\hours by60 \multiply\hours by 60 \advance\mins by-\hours
\divide\hours by60

\def\now{ \ifnum\hours>11 \ifnum\hours>12 \advance\hours by
-12 \fi \number\hours:\ifnum\mins<10 0\fi \number\mins\ pm,\ \else
\ifnum\hours=0 \hours=12 \fi \number\hours:\ifnum\mins<10 0\fi
\number\mins\ am,\ \fi}

\newcommand{\V}{\mathbb{V}}

\newcommand{\bea}{\begin{equation}}
\newcommand{\eea}{\end{equation}}

\newcommand{\CSK}{Cauchy-Stieltjes Kernel (CSK)\ \renewcommand{\CSK}{CSK\ }}
\newcommand{\FEF}{Free Exponential  (FE)\ \renewcommand{\FEF}{FE\ }}

\title{ Variance function of boolean additive convolution.}

\author{Raouf Fakhfakh }
\address{ Laboratory of Probability and Statistics, Faculty of Science,
Sfax University,  Sfax, Tunisia.}
\email{fakhfakh.raouf@gmail.com}

\begin{document}

\begin{abstract}Suppose $\V_{\nu}$ is the pseudo-variance function of the \CSK family ${\mathcal{K}_{+}}(\nu)$ generated by a non degenerate  probability measure $\nu$  with support bounded from above. We determine the formula for pseudo-variance function (or variance function $V_{\nu}$ in case of existence) under boolean additive convolution power.  This formulas is used to identify the relation between variance functions under Boolean Bercovici-Bata Bijection between probability measures. We also gives the connection between boolean cumulants and variance function and we relate  boolean cumulants of some probability measures to Catalan numbers and Fuss Catalan numbers.
\end{abstract}

\maketitle

\section{Introduction}

According to Weso{\l}owski \cite{Wesolowski99}, the kernel families generated by a kernel $k(x,\theta)$ with generating measure $\nu$ is the set of probability measures
$$\left\{k(x,\theta)/L(\theta): \ \theta\in \Theta\right\},$$
where $L(\theta)=\int k(x,\theta)\nu(dx)$ is the normalizing constant and $\nu$ is the generating measure. The theory of natural exponential families (NEFs) is based on the exponential kernel $k(x,\theta)=\exp(\theta x)$. The theory of \CSK families is recently introduced and it arises from a procedure analogous to the definition NEFs by using the Cauchy-Stieltjes kernel $1/(1-\theta x)$ instead of the exponential kernel. Bryc \cite{Bryc-06-08} initiated the study of  \CSK families for compactly supported probability measures $\nu$. He has shown that such families can be parameterized by the mean and under this parametrization, the family (and measure $\nu$) is uniquely determined by the variance function $V(m)$ and the mean $m_0$ of $\nu$.
 He also describe the class of quadratic \CSK families. This class consists of the free Meixner distributions. A formula for variance function under free additive convolution power is also given: that is for $\alpha>0$,
\begin{equation}\label{Vbox+alpha}
V_{\nu^{\boxplus \alpha}}(m)=\alpha V_{\nu}(m/\alpha).
\end{equation}
 In \cite{Bryc-Hassairi-09}, Bryc and Hassairi  extend the results in \CSK families to allow measures $\nu$ with unbounded support, providing the method to determine the domain of means, introducing  the ``pseudo-variance" function that has no direct probabilistic interpretation but has similar properties to the variance function. They have also introduced the notion of reciprocity between two \CSK families by defining a relation between the $R$-transforms of the corresponding generating probability measure. This leads to describing a class of cubic \CSK families (with support bounded from one side) which is related to quadratic class by a relation of reciprocity.  A general description of polynomial variance function with arbitrary degree is given in \cite{Bryc-Raouf-Wojciech-18}. In particular, a complete resolution of cubic compactly supported \CSK families is given. Other properties and characterizations in \CSK families regarding the mean of the reciprocal and orthogonal polynomials are also given in \cite{Raouf-17} and \cite{Raouf-18}.

On the other hand, authors in \cite{Speicher-Woroudi-93} introduce a new kind of convolution between probability measures in the context of non-commutative probability theory with boolean independence: the boolean additive convolution $\uplus$. It plays an important role in free probability theory. However, it gained
prominence in works by Biane, Belinschi and their collaborators (\cite{Bercovici-Pata-Biane-99}), as it turned out to be useful in analyzing
properties of the much more important free convolution.

In this paper, we deal with additive boolean convolution from a point of view related to \CSK families. The variance function is the fundamental concept for NEFs and also for \CSK families. So, after a review of some facts regarding \CSK families, we determine in section 2 the formula for variance function under additive boolean convolution power. This formula is used to identify the relation between variance functions under Boolean Bercovici-Bata Bijection between probability measures. In section 3, we give the connection between boolean cumulants and variance function and we relate  boolean cumulants of some probability measures to Catalan numbers and Fuss Catalan numbers.

\section{Variance function and boolean additive convolution $\uplus$.}
\subsection{\CSK families: Preliminaries and notations}
The notations used in what follows are the ones used in \cite{Bryc-Hassairi-09}, \cite{Bryc-Raouf-Hassairi-14}, \cite{Bryc-Raouf-Wojciech-18}, \cite{Raouf-18} and \cite{Raouf-17}.

\textit{Definition of \CSK families.} Let $\nu$ be a non-degenerate probability measure
with support bounded from above. Then
\begin{equation}   \label{M(theta)}
M_{\nu}(\theta)=\int \frac{1}{1- \theta x} \nu(dx)
\end{equation}
 is well defined for all $\theta\in [0,\theta_+)$ with $1/\theta_+=\max\{0, \sup {\rm supp} (\nu)\}$ and
 $$\mathcal{K}_+(\nu)=\{P_{(\theta,\nu)}(dx)\ :\  \theta\in(0,\theta_+)\}=\{Q_{(m,\nu)}(dx)\ :\  m\in(m_0(\nu),m_+(\nu))\} $$
is the one-sided \CSK family generated by $\nu$. That is,
$$P_{(\theta,\nu)}(dx)=\frac{1}{M_{\nu}(\theta)(1-\theta
x)}\nu(dx)$$ and  $Q_{(m,\nu)}(dx)$ is the corresponding parametrization by
the mean, which is given by $ Q_{(m,\nu)}(dx)=f_{\nu}(x,m)\nu(dx),$ with
\begin{equation} \label{F(V)}
f_{\nu}(x,m):=\left\{
               \begin{array}{ll}
                 \frac{\V_{\nu}(m)}{\V_{\nu}(m)+m(m-x)}, \ \ \ \ m\neq 0& \hbox{;} \\
                 1, \ \ \ \ \ \ \ \ \ \ \ \ \ \ \ \ \ \ \ \ m= 0, \ \ \V_{\nu}(0)\neq 0 & \hbox{;} \\
                 \frac{\V'_{\nu}(0)}{\V'_{\nu}(0)-x}, \ \ \ \ \ \ \ \ \ \ \ \ m= 0, \ \ \V_{\nu}(0)= 0 & \hbox{.}
               \end{array}
             \right.
\end{equation}
and  which involves the {\em pseudo-variance function} $\V_{\nu}(m)$ introduced later on.

\textit{Domain of means.}
The interval
$(m_0(\nu),m_+(\nu))$ is called the (one sided) domain of means, and is
determined as the image of $(0,\theta_+)$ under the strictly
increasing function $k_{\nu}(\theta)=\int x P_{(\theta,\nu)}(dx)$ which is given
by the formula
\begin{equation}\label{L2m}
k_{\nu}(\theta)=\frac{M_{\nu}(\theta)-1}{\theta M_{\nu}(\theta)}.
\end{equation}
From \cite[Remark 3.3]{Bryc-Hassairi-09}  we read out the following:
for a non-degenerate probability measure $\nu$ with support bounded
from above, define
 \begin{equation}\label{B(nu)}
 B(\nu)=\max\{0,\sup {\rm supp}(\nu)\}=1/\theta_+\in[0,\infty).
 \end{equation}
 Then the one-sided domain of means $(m_0(\nu),m_+(\nu))$ of  is
determined from the following formulas
\begin{equation}\label{m0}
m_0(\nu)=\lim_{\theta\to 0^+} k_{\nu}(\theta)
\end{equation}
and for $B=B(\nu)$,
\begin{equation}\label{m+}
m_+(\nu)=B-\lim_{z\to B^+}\frac{1}{G_{\nu}(z)},
\end{equation}
with Cauchy transform $G_{\nu}(.)$ given by
\begin{equation}\label{G-transform}
G_\nu(z)=\int\frac{1}{z-x}\nu(dx).
\end{equation}
 \begin{remark}
One may define the one-sided \CSK family for a
generating measure with support bounded from below. Then the
one-sided \CSK family ${\mathcal{K}}_-(\nu)$ is defined for
$\theta_-<\theta<0$, where $\theta_-$ is either $1/A(\nu)$ or
$-\infty$ with $A=A(\nu)=\min\{0,\inf supp(\nu)\}$.
In this case, the domain of the means for ${\mathcal{K}}_-(\nu)$ is the interval $(m_-(\nu),m_0(\nu))$ with $m_-(\nu)=A-1/G_\nu(A)$. If $\nu$ has compact support, the natural domain for the
parameter $\theta$ of the two-sided \CSK family
$\mathcal{K}(\nu)=\mathcal{K}_+(\nu)\cup\mathcal{K}_-(\nu)\cup\{\nu\}$
is $\theta_-<\theta<\theta_+$.
 \end{remark}
\textit{Variance and pseudo-variance functions}. The variance function
\begin{equation}
  \label{Def Var}
  V_{\nu}(m)=\int (x-m)^2 Q_{(m,\nu)}(dx)
\end{equation}
 is the fundamental concept of the theory of NEF, and also of the theory of \CSK families as presented in \cite{Bryc-06-08}. Unfortunately, if $\nu$ does not have the first moment (which is the case for 1/2-stable law), all measures in the \CSK family generated by $\nu$ have infinite variance. Therefore, authors in \cite{Bryc-Hassairi-09} introduce the pseudo-variance function $\V_{\nu}(.)$. It is easy to describe explicitly  if
$m_0(\nu)=\int x d\nu$ is finite. Then, see \cite[Proposition3.2]{Bryc-Hassairi-09} we know that
\begin{equation}\label{VV2V}
\frac{\V_{\nu}(m)}{m}=\frac{V_{\nu}(m)}{m-m_0}.\end{equation}
 In particular, $\V_{\nu}=V_{\nu}$ when $m_0(\nu)=0$. In general,
\begin{equation}\label{m2v}
\frac{\V_{\nu}(m)}{m}=\frac{1}{\psi_{\nu}(m)}-m,
\end{equation}
where $\psi_{\nu}:(m_0(\nu),m_+(\nu))\to (0,\theta_+)$ is the inverse of the
function $k_{\nu}(\cdot)$.
The generating measure $\nu$ is determined uniquely by the
 pseudo-variance function $\V_{\nu}$ through the following identities (for technical details, see \cite{Bryc-Hassairi-09}).
if
\begin{equation}\label{z2m}
 z=z(m)=m+\frac{\V_{\nu}(m)}{ m}
 \end{equation}
then the Cauchy transform satisfies
 \begin{equation}
  \label{G2V}
  G_\nu(z)=\frac{m}{\V_{\nu}(m)}.
\end{equation}

 \textit{The effects of affine transformations}. Here we collect formulas that describe the effects on the corresponding \CSK family of applying
 an affine transformation to the generating measure. For $\delta\neq0$ and $\gamma\in\mathbb{R}$, let $f(\nu)$ be the
 image of $\nu$ under the affine map $x\longmapsto(x-\gamma)/\delta$. In other words, if $X$ is a
 random variable with law $\nu$ then $f(\nu)$ is the law of
 $(X-\gamma)/\delta$, or
 $f(\nu)=D_{1/\delta}(\nu\boxplus\delta_{-\gamma})$, where
 $D_r(\mu)$ denotes the dilation of measure $\mu$ by a number
 $r\neq0$, i.e. $D_r(\mu)(U)=\mu(U/r)$.\\
 The effects of the affine transformation on the corresponding \CSK
 family are as follows :

 $\bullet$ Point $m_0$ is transformed to $(m_0-\gamma)/\delta$.
 In particular, if $\delta<0$, then $f(\nu)$ has support bounded
 from below and then it generates the left-sided $\mathcal{K}_-(f(\nu))$.

 $\bullet$ For $m$ close enough to $(m_0-\gamma)/\delta$ the pseudo-variance function is
 \begin{equation}\label{pseudo-var affine transf}
 \mathbb{V}_{f(\nu)}(m)=\frac{m}{\delta(m\delta+\gamma)}\mathbb{V}_\nu(\delta m+\gamma).
 \end{equation}
 In particular, if the variance function exists, then
 $$V_{f(\nu)}(m)=\frac{1}{\delta^2}V_\nu(\delta m+\gamma).$$
 A special case worth noting is the reflection $f(x)=-x$. If
 $\nu$ has support bounded from above and its right-sided \CSK
 family $\mathcal{K}_+(\nu)$ has domain of means $(m_0,m_+)$ and
 pseudo-variance function $\mathbb{V}_\nu(m)$, then $f(\nu)$
 generates the left-sided \CSK family $\mathcal{K}_-(f(\nu))$
 with domain of means $(-m_+,-m_0)$ and the pseudo-variance
 function $\mathbb{V}_{f(\nu)}(m)=\mathbb{V}_\nu(-m)$.

\subsection{Variance function under boolean additive convolution power.}
For a probability measure $\nu$ on $\mathbb{R}$, its Cauchy transform $G_{\nu}$ is defined by \eqref{G-transform}. Note that $G_{\nu}$ maps $\mathbb{C}^+$ to $\mathbb{C}^-$. The Boolean additive convolution is determined by the $K$-transform $K_{\nu}$ of $\nu$ which is defined as
\begin{equation}\label{Knu}
K_{\nu}(z)=z-\frac{1}{G_{\nu}(z)}, \ \ \ \ for \ z\in \mathbb{C}^+.
\end{equation}
The function $K_{\nu}$ is usually called self energy and it represent the analytic backbone of boolean additive convolution. For two probability measures $\mu$ and $\nu$, the additive Boolean convolution $\mu\uplus\nu$ is determined by
\begin{equation}\label{Kuplus}
K_{\mu\uplus\nu}(z)=K_{\mu}(z)+K_{\nu}(z), \ \ \ \ for \ z\in \mathbb{C}^+,
\end{equation}
and $\mu\uplus\nu$ is again a probability measure.\\
 The following result list properties of $K$-transform that we need.
\begin{proposition}\label{proposition}
Suppose $\V_{\nu}$ is the pseudo-variance function of the \CSK family ${\mathcal{K}_{+}}(\nu)$ generated by a non degenerate  probability measure $\nu$ with $b= \sup\rm supp (\nu)<\infty$. Then
\begin{itemize}
 \item[(i)]   $K_{\nu}$ is strictly decreasing on $(b,+\infty)$.
 \item[(ii)] For $m\in (m_0(\nu),m_+(\nu))$
\begin{equation}\label{KtransformV}
K_{\nu}\left(m+\V_{\nu}(m)/m\right)=m.
\end{equation}
\item[(iii)] $\displaystyle\lim_{z\longrightarrow B^+}K_{\nu}(z)=m_+(\nu)$, with $B=B(\nu)$ given by \eqref{B(nu)}.
\item[(iv)] $\displaystyle\lim_{z\longrightarrow +\infty}K_{\nu}(z)=m_0(\nu)\geq -\infty.$
\end{itemize}

\end{proposition}
\begin{proof}
(i) 
For a probability measure $\nu$ with support in $(-\infty, b]$, $G_{\nu}$ is analytic on the
slit complex plane $\mathbb{C} \setminus (-\infty, b]$. This shows that a probability measure $\nu$ with
support bounded from above is determined uniquely by $G_{\nu}(z)$ on $z\in(b,+\infty)$ for some $b$. This, together with \eqref{Knu} shows that a probability measure $\nu$ with support bounded from above is determined uniquely by $K_{\nu}(z)$ on $z\in(b,+\infty)$. From \cite[Proposition 3.7]{Bryc-Hassairi-09}, we have for $b<z_1<z_2$
\begin{equation}\label{Gz1z2}
G_{\nu}(z_1)-G_{\nu}(z_2)>(z_2-z_1)G_{\nu}(z_1)G_{\nu}(z_2).
\end{equation}
Then, \eqref{Gz1z2} say $1/G_{\nu}(z_2)-1/G_{\nu}(z_1)>z_2-z_1$. This implies that $K_{\nu}(z_2)-K_{\nu}(z_1)<0$.

(ii) From \cite[Proposition 3.3]{Bryc-Hassairi-09}, let $$z=z(m)=m+\V_{\nu}(m)/m.$$ We have that $m\longmapsto z(m)$ is continuous strictly decreasing on $(m_0,m_+)$, $z(m)>0$ on $(m_0,m_+)$, $z(m)\nearrow\infty$ as $m\searrow m_0$ and $z(m) \searrow B$ as $m\nearrow m_+$. In addition from \cite[Proposition 3.5]{Bryc-Hassairi-09}, the Cauchy transform satisfies \eqref{G2V}.
For $m\in (m_0(\nu),m_+(\nu))$, equations \eqref{Knu} and \eqref{G2V} implies that
$$K_{\nu}\left(m+\V_{\nu}(m)/m\right)=m+\V_{\nu}(m)/m-1/G_{\nu}\left(m+\V_{\nu}(m)/m\right)=m.$$
(iii) From \eqref{Knu} and \eqref{m+} (se also \cite[Proposition 3.4]{Bryc-Hassairi-09}) we see that
$$\displaystyle\lim_{z\longrightarrow B^+}K_{\nu}(z)=\displaystyle\lim_{z\longrightarrow B^+}z-1/G_{\nu}(z)=B-\lim_{z\to B^+}1/G_{\nu}(z)=m_+(\nu).$$

(iv) Since the limit exists by part (i), this is a consequence of \eqref{KtransformV}, in fact: $$\displaystyle\lim_{z\longrightarrow +\infty}K_{\nu}(z)=\displaystyle\lim_{m\longrightarrow m_0}K_{\nu}(m+\V_{\nu}(m)/m)=m_0(\nu).$$
\end{proof}
Denote by $\mathcal{M}$ the space of Borel probability measures on $\mathbb{R}$. According to \cite{Speicher-Woroudi-93}, we call a probability measure $\nu \in \mathcal{M}$ is infinitely divisible in the boolean sense, if for each $n\in\mathbb{R}$, there exists $\nu_n\in \mathcal{M}$ such that
$$\nu=\underbrace{\nu_n \uplus.....\uplus \nu_n}_{ n \  \mbox{times}}.$$
Note that all probability measure $\nu \in \mathcal{M}$ are $\uplus$-infinitely divisible, see \cite[Theorem 3.6]{Speicher-Woroudi-93}.
Next, we determine the formula for pseudo-variance function (and variance function $V_{\nu}$ in case of existence) under boolean additive convolution power.
\begin{theorem}\label{TH1}
Suppose $\V_{\nu}$ is the pseudo-variance function of the \CSK family ${\mathcal{K}_{+}}(\nu)$ generated by a non degenerate  probability measure $\nu$  with $b= \sup\rm supp (\nu)<\infty$ and mean $m_0(\nu)$. For $\alpha>0$, we have that:
\begin{itemize}
 \item[(i)] The support of $\nu^{\uplus \alpha}$ is bounded from above.
 \item[(ii)] For $m$ close enough to $\alpha m_0(\nu)$,
\begin{equation}\label{pseudovarnualpha1}
\V_{\nu^{\uplus \alpha}}(m)=\alpha\V_{\nu}(m/\alpha)+m^2(1/\alpha-1).
\end{equation}
Furthermore, if $m_0<+\infty$, then the variance functions of the \CSK families generated by $\nu$ and $\nu^{\uplus \alpha}$ exists and
\begin{equation}\label{varnualpha1}
V_{\nu^{\uplus \alpha}}(m)=\alpha V_{\nu}(m/\alpha)+m(m-\alpha m_0)(1/\alpha-1).
\end{equation}
\end{itemize}
\end{theorem}
\begin{proof}
(i) For measure $\nu$ with support
bounded from above by $b > 0$, $G_{\nu}$ is analytic on the
slit complex plane $\mathbb{C} \setminus (-\infty, b]$. We have that $\{x\in (supp\ \nu)^c; \ G_{\nu}(x)\neq 0 \}\subset (supp\ \nu)^c$ (see \cite[Lemma 2.1]{Takahiro Hasebe-10}. This implies that $K_{\nu}$ and so $K_{\nu^{\uplus  \alpha}}$ are well defined on a subset of  $(supp\ \nu)^c$.
So $G_{\nu^{\uplus  \alpha}}(z)$ is well defined and analytic on a subset of $(b,+\infty)$. Then the support of $\nu^{\uplus  \alpha}$ is bounded from above.

(ii) We see that  $m_0(\nu^{\uplus  \alpha})=\alpha m_0(\nu)$. This follows from proposition \ref{proposition}(iv) and the additive property of the K-transform. So for $m$ close enough to $\alpha m_0(\nu)$ so that $m/\alpha\in(m_0(\nu),m_+(\nu))$ and $m+\V_{\nu^{\uplus \alpha}}(m)/m\in(b,+\infty)$, we can apply \eqref{KtransformV} and the additive property of the K-transform to see that
 $$K_{\nu}\left(m+\frac{\V_{\nu^{\uplus \alpha}}(m)}{m}\right)=\frac{1}{\alpha}K_{\nu^{\uplus \alpha}}\left(m+\frac{\V_{\nu^{\uplus \alpha}}(m)}{m}\right)=\frac{m}{\alpha}=K_{\nu}\left(m/\alpha+\frac{\V_{\nu}(m/\alpha)}{m/\alpha}\right).$$
 Since $K_{\nu}$ is strictly decreasing on $(b,+\infty)$, this implies that
 $$m+\frac{\V_{\nu^{\uplus \alpha}}(m)}{m}=m/\alpha+\frac{\V_{\nu}(m/\alpha)}{m/\alpha},$$
 which is nothing but \eqref{pseudovarnualpha1}. Furthermore, if $m_0<+\infty$, then the variance functions of the \CSK families  ${\mathcal{K}_{+}}(\nu)$ and ${\mathcal{K}_{+}}(\nu^{\uplus \alpha})$ exists and relation \eqref{varnualpha1} follows from \eqref{VV2V} and \eqref{pseudovarnualpha1}.
\end{proof}

\begin{remark}
Let $\nu=\frac{1}{2}\delta_{-1}+\frac{1}{2}\delta_1$ be the symmetric Bernoulli distribution, its Cauchy transform and self energy are respectively
 $$G_{\nu}(z)=\frac{z}{z^2-1} \ \ \ \ \ \ \ \mbox{and} \ \ \ \ \ \ \  K_{\nu}(z)=\frac{1}{z}.$$ With $B(\nu)=\max\{0,\ \sup supp(\nu)\}=1$, we have  from Proposition \ref{proposition}(iii) $$m_+(\nu)=\lim_{z\longrightarrow 1}K_{\nu}(z)=1.$$ Consider $\mu=\nu^{\uplus 2}$, then we have $K_{\mu}(z)=K_{\nu^{\uplus 2}}(z)=2K_{\nu}(z)=2/z$ and $G_{\mu}(z)=\frac{z}{z^2-2}$. So $\mu=\frac{1}{2}\delta_{-\sqrt{2}}+\frac{1}{2}\delta_{\sqrt{2}}$. With $B(\mu)=\max\{0,\ \sup supp(\mu)\}=\sqrt{2}$, we have that $$m_+(\mu)=\lim_{z\longrightarrow \sqrt{2}}K_{\mu}(z)=\sqrt{2}.$$ This implies that $m_+(\nu^{\uplus 2})\neq 2m_+(\nu)$. So there is no "simple formula"
for $m_+$ under additive boolean convolution power. For this reason, in theorem \ref{TH1} we restrict ourself to $m$ close enough to $\alpha m_0(\nu)$.
\end{remark}
\begin{proposition}\label{proposition2}
Suppose $\V_{\nu}$ is the pseudo-variance function of the \CSK family ${\mathcal{K}_{+}}(\nu)$ generated by a non degenerate  probability measure $\nu$ with $b= \sup\rm supp (\nu)<\infty$ and mean $m_0(\nu)$, then for $\alpha> 0$  measure
\begin{equation}\label{nualpha}
\nu_{\alpha}:=D_{1/\alpha}(\nu^{\uplus\alpha})
\end{equation}
has also support bounded from above and there is $\varepsilon>0$ such that the pseudo-variance function of the one sided \CSK family generated by $\nu_{\alpha}$ is
$$\V_{\nu_{\alpha}}(m)=\V_{\nu}(m)/\alpha+(1/\alpha-1)m^2,$$
for all $m\in(m_0,m_0+\varepsilon)$.
\end{proposition}
\begin{proof}
It follows easily from  \eqref{pseudovarnualpha1} and \eqref{pseudo-var affine transf}.
\end{proof}
The following result gives  formulas for pseudo-variance functions (and variance functions in case of existence) under both free additive convolution and boolean additive convolution power.
\begin{proposition}
Suppose $\V_{\nu}$ is the pseudo-variance function of the \CSK family ${\mathcal{K}_{+}}(\nu)$ generated by a non degenerate  probability measure $\nu$ with support bounded from above. For $\alpha>0$ such that probability measures $\left(\nu^{\boxplus 1/\alpha}\right)^{\uplus \alpha}$ and $\left(\nu^{\uplus 1/\alpha}\right)^{\boxplus \alpha}$ are well defined, their support are bounded from above and they generates \CSK families with pseudo-variance functions
\begin{equation}\label{V1}
\V_{\left(\nu^{\boxplus 1/\alpha}\right)^{\uplus \alpha}}(m)=\V_{\nu}(m)+(1/\alpha-1)m^2,
\end{equation}
and
\begin{equation}\label{V2}
\V_{\left(\nu^{\uplus 1/\alpha}\right)^{\boxplus \alpha}}(m)=\V_{\nu}(m)+(1-1/\alpha)m^2.
\end{equation}
respectively, for $m$ close enough to $m_0$. Furthermore, if $m_0<+\infty$, then the variance functions of the \CSK families generated  respectively by $\nu$, $\left(\nu^{\boxplus 1/\alpha}\right)^{\uplus \alpha}$ and $\left(\nu^{\uplus 1/\alpha}\right)^{\boxplus \alpha}$ exists and for $m$ close enough to $m_0$ we have
\begin{equation}\label{v1}
V_{\left(\nu^{\boxplus 1/\alpha}\right)^{\uplus \alpha}}(m)=V_{\nu}(m)+(1/\alpha-1)m(m-m_0).
\end{equation}
and
\begin{equation}\label{v2}
V_{\left(\nu^{\uplus 1/\alpha}\right)^{\boxplus \alpha}}(m)=V_{\nu}(m)+(1-1/\alpha)m(m-m_0).
\end{equation}
\end{proposition}
\begin{proof}
If $\nu$ is a probability measure with support bounded from above, then $\nu^{\boxplus 1/\alpha}$, for $\alpha>0$, has also support bounded from above (see \cite[Proposition 3.10]{Bryc-Hassairi-09} or \cite[Proposition 6.1]{Bercovici-Voiculescu-93}). This implies that $\left(\nu^{\boxplus 1/\alpha}\right)^{\uplus \alpha}$ has support bounded from above. In addition
\begin{eqnarray*}
\V_{\left(\nu^{\boxplus 1/\alpha}\right)^{\uplus \alpha}}(m)
 & = & \alpha\V_{\nu^{\boxplus 1/\alpha}}(m/\alpha)+(1/\alpha-1)m^2\\
 & = & \V_{\nu}(m)+(1/\alpha-1)m^2.
 \end{eqnarray*}
 Furthermore, if $m_0<+\infty$, then the variance functions of the \CSK families  ${\mathcal{K}_{+}}(\nu)$ and ${\mathcal{K}_{+}}\left(\left(\nu^{\boxplus 1/\alpha}\right)^{\uplus \alpha}\right)$ exists and relation \eqref{v1} follows from \eqref{VV2V} and \eqref{V1}.

 The same arguments are used for probability measure $\left(\nu^{\uplus 1/\alpha}\right)^{\boxplus \alpha}$ to get formulas \eqref{V2} and \eqref{v2}.
\end{proof}
\subsection{Relation with Boolean Bercovici-Bata Bijection.}
Authors in \cite{Belinschi-Nica-08} consider the transformation $\mathbb{B}_t:\mathcal{M}\longmapsto \mathcal{M}$ defined by, for every $t\geq0$
\begin{equation}\label{Btmu}
\mathbb{B}_t(\mu)=\left(\mu^{\boxplus (1+t)}\right)^{\uplus \frac{1}{1+t}}, \ \ \ \ \ \mu\in \mathcal{M}.
\end{equation}
 They prove that for $t = 1$ the transformation $\mathbb{B}_1$ coincides with the canonical bijection $\mathbb{B} : \mathcal{M} \longmapsto \mathcal{M}_{Inf-div}$ discovered by Bercovici and Pata in their study of the relations between infinite divisibility in free and in Boolean probability. Here $\mathcal{M}_{Inf-div}$ stands for the set of probability distributions in $\mathcal{M}$ which are infinitely divisible with respect to the operation $\boxplus$. As a consequence, we have that $\mathbb{B}_t(\mu)$ is $\boxplus$-infinitely divisible for every $\mu\in \mathcal{M}$ and every $t \geq 1$. The following result gives the pseudo-variance function (and variance function in case of existence) of $\mathbb{B}_t(\mu)$. In fact this easily follows from \eqref{V1} and \eqref{v1} by choosing $\alpha=\frac{1}{1+t}$.

\begin{proposition}
Suppose $\V_{\nu}$ is the pseudo-variance function of the \CSK family ${\mathcal{K}_{+}}(\nu)$ generated by a non degenerate  probability measure $\nu$ with support bounded from above. For $t\geq0$, the probability measure
\begin{equation}\label{Btmu}
\mathbb{B}_t(\nu)=\left(\nu^{\boxplus (1+t)}\right)^{\uplus \frac{1}{1+t}}
\end{equation}
has support bounded from above and it generates the \CSK family with pseudo-variance function
\begin{equation}\label{VBtmu}
\V_{\mathbb{B}_t(\nu)}(m)=\V_{\nu}(m)+tm^2.
\end{equation}
Furthermore, if $m_0<+\infty$, then the variance functions of the \CSK families generated by $\nu$ and $\mathbb{B}_t(\nu)$ exists and
\begin{equation}\label{vBtmu}
V_{\mathbb{B}_t(\nu)}(m)=V_{\nu}(m)+tm(m-m_0).
\end{equation}
\end{proposition}
Denote by  $\mathcal{V}$ the class of variance functions corresponding to
probability measures $\nu$ such that $\nu$ is compactly supported, centered: $\int x\nu(dx)=0$, with variance  $\int
x^2\nu(dx)=1$, so that $V_{\nu}(0)=1$.
Denote by $\mathcal{V}_{\infty}$ the class of those $V\in\mathcal{V}$ that the function
 $m\mapsto V(cm)$  is in $\mathcal{V}$ for every real $c$.
It was proved in \cite[Corollary 1.1]{Bryc-Raouf-Wojciech-18}, that the map $V(m)\mapsto V(m)-m^2$ is a bijection
of $\mathcal{V}_{\infty}$ onto $\mathcal{V}$ (also the map $V_{\nu}(m)\longmapsto V_{\nu}(m)+m^2$ is the inverse bijection). We will see that this bijection between variance functions coincide with the boolean Bercovici-Bata bijection.
\begin{proposition}
Suppose $V_{\nu}(.)$ is the variance function of the \CSK family  generated by a non degenerate  probability measure $\nu$ with mean 0 and variance 1. For $\alpha>0$ such that probability measures $\left(\nu^{\boxplus 1/\alpha}\right)^{\uplus \alpha}$ and $\left(\nu^{\uplus 1/\alpha}\right)^{\boxplus \alpha}$ are well defined, we have
\begin{itemize}
 \item[(i)] The bijection $V_{\nu}(m)\longmapsto V_{\nu}(m)+m^2$ from $\mathcal{V}$  onto  $\mathcal{V}_{\infty}$ correspond to Boolean Bercovici-Bata bijection  between probability measures $\nu\longmapsto \mathbb{B}_1(\nu)$, in addition
 \begin{equation}\label{limit1}
 \left(\nu^{\uplus 1/\alpha}\right)^{\boxplus \alpha} \xrightarrow{\alpha\to +\infty} \mathbb{B}_1(\nu), \ \ \ \ \mbox{in distribution}.
 \end{equation}
  \item[(ii)] The bijection $V_{\nu}(m)\longmapsto V_{\nu}(m)-m^2$ from $\mathcal{V}_{\infty}$ onto $\mathcal{V}$  correspond to the inverse Boolean Bercovici-Bata bijection  between probability measures $\nu\longmapsto \mathbb{B}_1^{-1}(\nu)$, in addition
 \begin{equation}\label{limit2}
 \left(\nu^{\boxplus 1/\alpha}\right)^{\uplus \alpha} \xrightarrow{\alpha\to +\infty} \mathbb{B}_1^{-1}(\nu), \ \ \ \ \mbox{in distribution}.
 \end{equation}
 \end{itemize}
\end{proposition}
\begin{proof}
(i) It is clear from \eqref{vBtmu} that the bijection between variance functions $V_{\nu}(m)\longmapsto V_{\nu}(m)+m^2$ is in fact the celebrated Boolean Bercovici-Bata bijection $\mathbb{B}_1(.)$. This bijection is given by the explicit formula between probability measures $\nu\longmapsto \mathbb{B}_1(\nu)=\left(\nu^{\boxplus 2}\right)^{\uplus \frac{1}{2}}$. On the other hand, one see from \eqref{v2} that
$$\displaystyle\lim_{\alpha\rightarrow +\infty}V_{\left(\nu^{\uplus 1/\alpha}\right)^{\boxplus \alpha}}(m)=V_{\nu}(m)+m^2=V_{\mathbb{B}_1(\nu)}(m),$$
which implies \eqref{limit1}.

(ii) The inverse Boolean Bercovici-Bata bijection is also explicit and takes two forms: 
the  bijection between variance functions $V_{\nu}(m)\longmapsto V_{\nu}(m)-m^2$ correspond to inverse Boolean Bercovici-Bata bijection between probability measures $\nu\longmapsto\mathbb{B}_1^{-1}(\nu)=\left(\nu^{\uplus 2}\right)^{\boxplus \frac{1}{2}}$. On the other hand, one see from \eqref{v1} that
$$\displaystyle\lim_{\alpha\rightarrow +\infty}V_{\left(\nu^{\boxplus 1/\alpha}\right)^{\uplus \alpha}}(m)=V_{\nu}(m)-m^2=V_{\mathbb{B}_1^{-1}(\nu)}(m),$$
which implies \eqref{limit2}.

\end{proof}
\section{Boolean cumulants and variance functions}
If $\nu$ is a compactly supported probability measure on the real line, the $K$-transform $K_{\nu}$ of $\nu$ admit a Laurent expansion. From \cite{Speicher-Woroudi-93}, one sees that
\begin{equation}\label{formula3}
K_{\nu}(z)=\sum_{n=1}^{\infty}r_n(\nu)\frac{1}{z^{n-1}}.
\end{equation}
The coefficients $r_n=r_n(\nu)$ are called the boolean cumulants of the measure $\nu$. In particular $r_0=0$, $r_1=\int x \nu(dx)=m_0$. The following result gives the connection between boolean cumulants and variance functions of \CSK families.
\begin{theorem}\label{proposition3}
Suppose $V_{\nu}$ is analytic in a neighborhood of $m_0$, $V_{\nu} (m_0) > 0$,
and $\nu$ is a probability measure with finite all moments, such that $\int x\nu(dx) =
m_0$. Then the following conditions are equivalent.
\begin{itemize}
 \item[(i)] $\nu$ is non degenerate, compactly supported and there exists an interval $(A,B)\ni m_0$ such that $\{Q_{(m,\nu)}(dx)=f_{\nu}(x,m)\nu(dx) \ :\  m\in(A,B)\}$, with $f_{\nu}(x,m)$ given by \eqref{F(V)}, define a family of probability measures parameterized by the mean with variance function $V_{\nu}(.)$
 \item[(ii)] The boolean cumulants of the measure $\nu$ are $r_0=0$, $r_1=m_0$ and for all $n\geq 1$
\begin{equation}\label{formula4}
r_{n+1}=\frac{1}{n!}\left.\frac{d^{n-1}}{dm^{n-1}}\left(V_{\nu}(m)+m(m-m_0)\right)^n\right|_{m=m_0}.
\end{equation}
\end{itemize}
\end{theorem}
\begin{proof}
$(i) \Rightarrow (ii) $ From Proposition \ref{proposition} (ii) and relation \eqref{VV2V}, one see that for all $m\in(m_0,m_+)$
\begin{equation}\label{formula5}
K_{\nu}(m+V_{\nu}(m)/(m-m_0))=m.
\end{equation}
The $L$-transformation $L(z)=K_{\nu}(1/z)$ has a Taylor expansion given by
\begin{equation}\label{formulaL}
L(z)=\sum_{n=1}^{\infty}r_n(\nu)z^{n-1}.
\end{equation}
On the other hand, the Lagrange expansion theorem says that if $\phi(z)$ is analytic in a neighborhood of $z=m_0$, $\phi(m_0)\neq 0$ and $\xi:=(m-m_0)/\phi(m)$ then
\begin{equation}\label{formulaL1}
L(\xi)=m_0+\sum_{n=1}^{\infty}\frac{\xi^n}{n!}\left.\frac{d^{n-1}}{dm^{n-1}}(\phi(m))^n\right|_{m=m_0}
\end{equation}
For $\phi(m)=V_{\nu}(m)+m(m-m_0)$, $\xi(m):=(m-m_0)/\phi(m)=\frac{1}{m+V_{\nu}(m)/(m-m_0)}$, equation \eqref{formulaL} is
\begin{equation}\label{formulaL2}
L(\xi(m))=\sum_{n=1}^{\infty}r_n(\nu)(\xi(m))^{n-1}.
\end{equation}
and equation \eqref{formulaL1} becomes
\begin{equation}\label{formulaL3}
L(\xi(m))=m_0+\sum_{n=1}^{\infty}\frac{(\xi(m))^n}{n!}\left.\frac{d^{n-1}}{dm^{n-1}}(V_{\nu}(m)+m(m-m_0))^n\right|_{m=m_0}
\end{equation}
An identification between \eqref{formulaL2} and \eqref{formulaL3} implies \eqref{formula4}.

$(ii) \Rightarrow (i) $ Using \eqref{vBtmu}, for $t=1$, the boolean cumulants of $\nu$ given by \eqref{formula4} can be written as
\begin{equation}\label{formula5}
r_{n+1}=\frac{1}{n!}\left.\frac{d^{n-1}}{dm^{n-1}}\left(V_{\mathbb{B}_1(\nu)}(m)\right)^n\right|_{m=m_0}.
\end{equation}
According to \cite[Theorem 3.3]{Bryc-06-08}, the boolean cumulants $r_{n+1}$ of $\nu$ are in fact the free cumulants of the probability measure $\mathbb{B}_1(\nu)$. This implies that $\mathbb{B}_1(\nu)$ is non degenerate, compactly supported and there exists an interval $(\alpha,\beta)\ni m_0$ such that $\{Q_{(m,\mathbb{B}_1(\nu))}(dx)=f_{\mathbb{B}_1(\nu)}(x,m)\mathbb{B}_1(\nu)(dx) \ :\  m\in(\alpha,\beta)\}$ define a family of probability measures parameterized by the mean with variance function $V_{\mathbb{B}_1(\nu)}(.)$. Applying the inverse Boolean Bercovici-Bata bijection to $\mathbb{B}_1(\nu)$, then the probability measure $\nu=\mathbb{B}_1^{-1}\left(\mathbb{B}_1(\nu)\right)$ is non degenerate, compactly supported and there exists an interval $(A,B)\ni m_0$ such that $\{Q_{(m,\nu)}(dx)=f_{\nu}(x,m)\nu(dx) \ :\  m\in(A,B)\}$ define a family of probability measures parameterized by the mean with variance function $V_{\nu}(.)$.
\end{proof}
 As pointed in the introduction the class of quadratic \CSK families is described in \cite[Theorem 3.2]{Bryc-06-08}. This class consists of the free Meixner distributions.

We now relate  boolean cumulants of the Marchenko Pastur distribution to Catalan numbers. The centered Marchenko-Pastur  distribution is given by
 $$\nu(dx)=\displaystyle\frac{\sqrt{4-(x-a)^2}}{2\pi(1+ax)}\textbf{1}_{(a-2,a+2)}(x)dx+p_1\delta_{x_1}.$$
 The discrete part is absent except for $a^2>1$, in this case $p_1=1-1/a^2$ and $x_1=-1/a$.
 It generates the  \CSK family with variance function
$V_{\nu}(m)=1+am={\V}_{\nu}(m)$.

\begin{corollary}
If $\nu$ is the centered standardized  Marchenko Pastur distribution with parameter $a=2$, i.e. it generates the \CSK family with $m_0=0$ and variance function $V_{\nu}(m)=1+2m$, then its boolean cumulants are $r_0=0$, $r_1=m_0=0$ and for $n\geq1$,
\begin{equation}
r_{n+1}(\nu)=\frac{1}{1+n}\left(
                       \begin{array}{c}
                         2n \\
                         n \\
                       \end{array}
                     \right).
\end{equation}
\end{corollary}
\begin{proof}
From \eqref{formula4}, for $a=2$, the boolean cumulants of $\nu$ are
$$r_{n+1}(\nu)=\frac{1}{n!}\left.\frac{d^{n-1}}{dm^{n-1}}(1+m)^{2n}\right|_{m=0}=\frac{2n(2n-1)(2n-2).....(n+2)}{n!}=\frac{1}{n+1}\frac{(2n)!}{(n!)^2}.$$
\end{proof}
Authors in \cite{Bryc-Raouf-Wojciech-18} introduce variance functions that are polynomial in the mean of arbitrary degree. In particular a complete resolution of compactly supported \CSK with cubic variance function is given (see \cite[Theorem 1.2]{Bryc-Raouf-Wojciech-18}): Fix $a,b,c\in\mathbb{R}$. A cubic function \begin{equation}\label{cubicV}
V(m)=1+am+bm^2 +cm^3
\end{equation}
 is in  $\mathcal{V}$ if and only if $(b+1)^3\geq 27 c^2$.
Furthermore,  $V$ is in $\mathcal{V}_\infty$ if and only if $b^3\geq 27 c^2$.

We now relate  boolean cumulants of certain probability distribution to Fuss-Catalan numbers.
\begin{corollary}
The function $V(m)=1+3m+2m^2+m^3$ is the variance function the \CSK family generated by a compactly supported probability measure $\nu$, with mean $0$, variance 1 and with boolean cumulants given by:
\begin{equation}
r_{n+1}(\nu)=\frac{1}{3n+1}\left(
                       \begin{array}{c}
                         3n+1 \\
                         n \\
                       \end{array}
                     \right).
\end{equation}
\end{corollary}
\begin{proof}
The function $V(m)=1+3m+2m^2+m^3$ correspond to \eqref{cubicV} with $a=3$, $b=2$ and $c=1$. In this case we have $(b+1)^3\geq 27 c^2$ and this implies that the function $V(m)$ is the variance function the \CSK family generated by a compactly supported probability measure $\nu$, with mean $0$ and variance 1. Furthermore,
from \eqref{formula4}, the boolean cumulants of $\nu$ are the Fuss Catalan numbers of order 3, that is
$$r_{n+1}(\nu)=\frac{1}{n!}\left.\frac{d^{n-1}}{dm^{n-1}}(1+m)^{3n}\right|_{m=0}=\frac{3n(3n-1)(3n-2).....(2n+2)}{n!}=\frac{1}{3n+1}\frac{(3n+1)!}{(2n+1)!(n!)}.$$
\end{proof}
Combining \eqref{formula4} with the $\uplus$-L\'{e}vy-Khinchin formula, we get also the following.
\begin{corollary}
Suppose $V(.)$ is analytic at $0$, $V(0)=1$. Then the following conditions are equivalent.
\begin{itemize}
 \item[(i)] There exists a   probability measure $\nu$ with mean 0 and variance 1 such
that $V(.)$ is the variance function of the \CSK family generated
by $\nu$.
 \item[(ii)] There exists a finite positive measure $\rho$ such that
 \begin{equation}\label{BC}
 \frac{1}{n!}\left.\frac{d^{n-1}}{dm^{n-1}}\left(V(m)+m^2\right)^n\right|_{m=0}=\int x^{n-1} \rho (dx), \ \ n\geq 1.
 \end{equation}
 \end{itemize}
\end{corollary}
\begin{proof}
From \cite[Proposition 3.2]{Speicher-Woroudi-93}, $K_{\nu}$ is a self energy function if and only if there exists a finite positive measure $\rho$ such that $K_{\nu}$ is the Cauchy transform of $\rho$: that is
\begin{equation}\label{KC}
K_{\nu}(z)= \int \frac{1}{z-x}\rho(dx).
\end{equation}
Combining \eqref{KC} with \eqref{formula4} and \eqref{formula3}, this end the proof.
\end{proof}


\begin{thebibliography}{}

\bibitem[1]{Bryc-06-08}
Bryc, W. (2009).
\newblock Free exponential families as kernel families.
\newblock {\em Demonstr. Math.}, XLII(3):657--672.
\newblock arxiv.org:math.PR:0601273.

\bibitem[2]{Bryc-Hassairi-09}
Bryc, W. and Hassairi, A. (2011).
\newblock One-sided {C}auchy-{S}tieltjes kernel families.
\newblock {\em Journ. Theoret. Probab.}, 24(2):577--594.
\newblock arxiv.org/abs/0906.4073.



\bibitem[3]{Bryc-Raouf-Hassairi-14}  Bryc, W. R. Fakhfakh and A. Hassairi.
\newblock On {C}auchy-{S}tieltjes kernel families.
\newblock {\em Journ. Multivariate. Analysis.}. 124: 295-312, 2014



\bibitem[4]{Bryc-Raouf-Wojciech-18}  Bryc, W. R. Fakhfakh and W. Mlotkowski.
\newblock  {C}auchy-{S}tieltjes families with polynomial variance funtions and generalized orthogonality.
\newblock {Probability and Mathematical Statistics} Vol. 39, Fasc. 2 (2019), pp. 237–258
doi:10.19195/0208-4147.39.2.1.

\bibitem[5]{Bercovici-Voiculescu-93} Bercovici, H., Voiculescu, D.
\newblock  Free convolution of measures with unbounded support.
\newblock {Indiana Univ. Math. J. 42(3), 733–773} (1993).

\bibitem[6]{Bercovici-Pata-Biane-99} Bercovici, H.,  Pata, V. Stable laws and domains of attraction in free probability
theory (with an appendix by P. Biane), Ann. of Math. (2) 149 (1999), 1023–1060.

\bibitem[7]{Raouf-17} R. Fakhfakh.
\newblock  The mean of the reciprocal in a {C}auchy-{S}tieltjes family.
\newblock {Statistics and Probability Letters} 129 (2017) 1–11.

\bibitem[8]{Raouf-18} R. Fakhfakh.
\newblock   Characterization of quadratic {C}auchy-{S}tieltjes Kernels families based on the orthogonality of polynomials.
\newblock {J. Math.Anal.Appl.} 459(2018)577–589.

\bibitem[9]{Speicher-Woroudi-93} R. Speicher, R. Woroudi.
\newblock  Boolean convolution.
\newblock {Fields Inst. Commun} 12 (1997) 267–279.

\bibitem[10]{Belinschi-Nica-08} Serban T. Belinschi, Alexandru Nica.
\newblock On a Remarkable Semigroup of Homomorphisms with Respect to Free Multiplicative Convolution.
\newblock {Indiana University Mathematics Journal}, Vol. 57, No. 4 (2008), pp. 1679-1713 459(2018)577–589.

\bibitem[11]{Takahiro Hasebe-10} Takahiro Hasebe.
\newblock Monotone convolution semigroups.
\newblock January 2010 Studia Mathematica 200(2) DOI: 10.4064/sm200-2-5


\bibitem[12]{Wesolowski99} Wesolowski, J.
\newblock  Kernels families.
\newblock {Unpublished manuscript} (1999).




\end{thebibliography}
\end{document}